\documentclass[11pt]{amsart}
\usepackage{amsmath,amssymb}
\newtheorem{theorem}{Theorem}

\newtheorem{proposition}{Proposition}

\newtheorem{lemma}{Lemma}
\theoremstyle{definition}

\newtheorem{question}{Question}
\newtheorem{rmk}{Remark}

\usepackage{latexsym,amssymb,mathrsfs}
\begin{document}
\baselineskip=14.5pt
\title[On the $p$-ranks of the ideal class groups of imaginary quadratic fields]{On the $p$-ranks of the ideal class groups of imaginary quadratic fields} 

\author{Jaitra Chattopadhyay and Anupam Saikia}
\address[Jaitra Chattopadhyay]{Department of Mathematics, Indian Institute of Technology, Guwahati, Guwahati - 781039, Assam, India}
\email[Jaitra Chattopadhyay]{jaitra@iitg.ac.in, chat.jaitra@gmail.com}
\address[Anupam Saikia]{Department of Mathematics, Indian Institute of Technology, Guwahati, Guwahati - 781039, Assam, India}
\email[Anupam Saikia]{a.saikia@iitg.ac.in}

\begin{abstract}
For a prime number $p \geq 5$, we explicitly construct a family of imaginary quadratic fields $K$ with ideal class groups $Cl_{K}$ having $p$-rank ${{\rm{rk}}_{p}(Cl_{K})}$ at least $2$. We also quantitatively prove, under the assumption of the $abc$-conjecture, that for sufficiently large positive real numbers $X$ and any real number $\varepsilon$ with $0 < \varepsilon < \frac{1}{p - 1}$, the number of imaginary quadratic fields $K$ with the absolute value of the discriminant $d_{K}$ $\leq X$ and ${{\rm{rk}}_{p}(Cl_{K})} \geq 2$ is $\gg X^{\frac{1}{p - 1} - \varepsilon}$. This improves the previously known lower bound of $X^{\frac{1}{p} - \varepsilon}$ due to Byeon and the recent bound $X^{\frac{1}{p}}/(\log X)^{2}$ due to Kulkarni and Levin.
\end{abstract}

\renewcommand{\thefootnote}{}



\renewcommand{\thefootnote}{\arabic{footnote}}
\setcounter{footnote}{0}

\maketitle

\section{Introduction}

The question of the divisibility of class numbers of quadratic fields has been an object of much study for a long time. Due to the works of Nagell \cite{nagell}, and later by Ankeny and Chowla \cite{AC}, it is known that there exist infinitely many imaginary quadratic fields $K$ with class numbers $h_{K}$ divisible by a given integer $n \geq 2$. Much later, Yamamoto \cite{yama}, Weinberger \cite{berger} and many others proved the analogue for real quadratic fields. There has also been much work done towards the quantitative study (cf. \cite{bilu-luca}, \cite{byeon1}, \cite{byeon2}, \cite{kalyan}, \cite{self-jrms}, \cite{hb1}, \cite{hb2}, \cite{kishi-miyake}, \cite{luca}, \cite{murty}, \cite{sound}, \cite{yu1}). Apart from these, the topic of indivisibility of class numbers (cf. \cite{olivia}, \cite{byeon}, \cite{self-rama}, \cite{kohnen-ono}, \cite{ono}) and simultaneous divisibility and indivisibility of class numbers (cf. \cite{self-acta}, \cite{iizuka}, \cite{komatsu-acta}, \cite{komatsu-ijnt}) received much attention in recent times.

\smallskip

Since the ideal class group $Cl_{K}$ of $K$ is a finite abelian group, it makes sense to consider the {\it $m$-rank} of $Cl_{K}$, for various integers $m \geq 2$. The $m$-rank of $Cl_{K}$, denoted by ${\rm{rk}}_{m}(Cl_{K})$, is defined to be the maximal integer $r \geq 0$ such that $(\mathbb{Z}/m\mathbb{Z})^{r} \subseteq Cl_{K}$. Note that if ${\rm{rk}}_{m}(Cl_{K}) \geq 1$, then it immediately follows that $h_{K}$ is divisible by $m$. Thus the study of the quantity ${\rm{rk}}_{m}(Cl_{K})$ sheds more light on the question of $m$-divisibility of $h_{K}$. In particular, the case $m = 3$ has been extensively studied (cf. \cite{craig1}, \cite{craig2}, \cite{kishi-komatsu}, \cite{levin}, \cite{luca-pacelli}, \cite{yu2}).

\smallskip

In this article, we are concerned about the ranks of the ideal class groups of imaginary quadratic fields. For integers $m \geq 2$, $n \geq 0$ and a large positive real number $X$, let $$N_{m}^{n}(X) = \{K = \mathbb{Q}(\sqrt{-d}) : d > 0, |d_{K}| \leq X \mbox{ and } {\rm{rk}}_{m}(Cl_{K}) \geq n\}$$ and we can ask the following question.

\begin{question}
What is the asymptotic behaviour of $N_{m}^{n}(X)$ as $X \to \infty$?
\end{question} 

In \cite{luca-pacelli}, Luca and Pacelli proved that $N_{3}^{2}(X) \gg X^{\frac{1}{3}}$. Yu \cite{yu2} recently improved that lower bound to $X^{\frac{1}{2} - \varepsilon}$. Also, Levin et al. \cite{levin} proved that $N_{3}^{5}(X) \gg \frac{X^{\frac{1}{30}}}{\log X}$. For an arbitrary integer $m$, the lower bound $N_{m}^{2}(X) \gg X^{\frac{1}{m} - \varepsilon}$ is known due to the work by Byeon \cite{byeon-rama}. Very recently, Kulkarni and Levin \cite{levin-arxiv} improved this and proved that $N_{m}^{2}(X) \gg X^{\frac{1}{m}}/(\log X)^{2}$. In the same paper, Kulkari and Levin improved the lower bound of $N_{3}^{5}(X)$ to $X^{\frac{1}{15}}/(\log X)^{2}$.

\smallskip

In this paper, under the assumption of the validity of the $abc$-conjecture, we slightly improve the lower bound for $N_{p}^{2}(X)$ for any prime number $p \geq 5$, by extending the method used by Yu \cite{yu2}. More precisely, we prove the following theorem.

\begin{theorem}\label{main-th}
Assume that the $abc$-conjecture holds true. Then for sufficiently large positive real numbers $X$, any positive real number $\varepsilon$ with $0 < \varepsilon < \frac{1}{p - 1}$ and any prime number $p \geq 5$, we have $N_{p}^{2}(X) \gg X^{\frac{1}{p - 1} - \varepsilon}$.
\end{theorem}

\section{A family of quadratic fields with ${\rm{rk}}_{p}(Cl_{K}) \geq 2$}

In what follows, $p$ always denotes a fixed prime number $\geq 5$. For two non-zero real-valued function $f$ and $g$, we use the notation $f \ll g$ or $f = O(g)$ to mean that $\frac{|f|}{|g|}$ is bounded and $f \asymp g$ to mean that both $f \ll g$ and $g \ll f$ hold. Also, $\mu$ stands for the M\"{o}bius function and $\sigma_{0}$ denotes the divisor function.

\smallskip

In this section, we prove a proposition that provides a parametric family of imaginary quadratic fields with ideal class groups having $p$-ranks at least $2$. Then by appropriately choosing the parameters, we shall count how many of these fields have the absolute values of their discriminants $\leq X$. Our following proposition is an extension of Lemma 2.1 of Yu \cite{yu2}.

\begin{proposition}\label{propn-1}
For a positive real number $X$, let $q$ be a prime number with $q \asymp X^{\frac{p - 2}{2p(p - 1)}}$. Let $a$ and $b$ be integers with $\frac{1}{2p}q^{\frac{p}{p - 2}} < a,b < (\frac{1}{2p} + \frac{1}{2^{p}p^{p}})q^{\frac{p}{p - 2}}$. Let $$f_{q}(a,b) = 2(a^{p} + b^{p})q^{p} - (a - b)^{2}q^{2p} - g(a,b)^{2},$$ where $g(a,b) = \frac{a^{p} - b^{p}}{a - b}$. If $X$ is sufficiently large and $f_{q}(a,b) \geq 4ab^{\frac{p - 1}{2}}q^{\frac{p + 1}{2}}$, and is square-free, then the ideal class group of the imaginary quadratic field $\mathbb{Q}(\sqrt{-f_{q}(a,b)})$ contains a subgroup isomorphic to $\mathbb{Z}/p\mathbb{Z} \times \mathbb{Z}/p\mathbb{Z}$.
\end{proposition}

\begin{proof}
Let $X_{1} = g(a,b) + (a - b)q^{p}$, $Y_{1} = aq$, $X_{2} = g(a,b) - (a - b)q^{p}$ and $Y_{2} = bq$. Then we see that 
\begin{equation}\label{eq-1}
X_{1}^{2} - 4Y_{1}^{p} = -f_{q}(a,b) = X_{2}^{2} - 4Y_{2}^{p}. 
\end{equation}

For the sake of brevity, let $D = f_{q}(a,b)$. Since $D$ is assumed to be square-free, we get from \eqref{eq-1} that $X_{j}$ is odd and therefore $-D \equiv 1 \pmod {4}$.  Also, for $j = 1,2$, from equation \eqref{eq-1}, we obtain $$\displaystyle\frac{X_{j} + \sqrt{-D}}{2}\cdot \displaystyle\frac{X_{j} - \sqrt{-D}}{2} = Y_{j}^{p} \mbox{ in } \mathcal{O}_{\mathbb{Q}(\sqrt{-D})}.$$ 

Hence $X_{j},Y_{j}^{p} \in \mathcal{I}_{j} :=\left\langle\displaystyle\frac{X_{j} + \sqrt{-D}}{2}, \displaystyle\frac{X_{j} - \sqrt{-D}}{2}\right\rangle$. Since $X_{j}^{2} - 4Y_{j}^{p} = -D$ is assumed to be square-free, we conclude that $\gcd(X_{j},Y_{j}) = 1$. Therefore, $\mathcal{I}_{j} = \mathcal{O}_{K}$ and consequently, $\left\langle\displaystyle\frac{X_{j} + \sqrt{-D}}{2}\right\rangle$ and $\left\langle\displaystyle\frac{X_{j} - \sqrt{-D}}{2}\right\rangle$ are comaximal ideals in $\mathcal{O}_{\mathbb{Q}(\sqrt{-D})}$. Thus there exists an integral ideal $\mathfrak{a}_{j} \subseteq \mathcal{O}_{\mathbb{Q}(\sqrt{-D})}$ such that $$\left\langle\displaystyle\frac{X_{j} + \sqrt{-D}}{2}\right\rangle = \mathfrak{a}_{j}^{p} \mbox{ and } \left\langle\displaystyle\frac{X_{j} - \sqrt{-D}}{2}\right\rangle = \overline{\mathfrak{a}_{j}}^{p}.$$

Now, we prove that the ideal classes $[\mathfrak{a}_{1}]$ and $[\mathfrak{a}_{2}]$ of both $\mathfrak{a}_{1}$ and $\mathfrak{a}_{2}$ generate subgroups of order $p$ inside $Cl_{\mathbb{Q}(\sqrt{-D})}$ and they intersect trivially. For that, it suffices to prove that $\mathfrak{a}_{1}$, $\mathfrak{a}_{2}$ and $\mathfrak{a}_{1}\mathfrak{b}^{t}$ are all non-principal ideals in $\mathcal{O}_{\mathbb{Q}(\sqrt{-D})}$, for all $t = 1,\ldots ,\frac{p - 1}{2}$ and for $\mathfrak{b} = \mathfrak{a}_{2} \mbox{ or } \overline{\mathfrak{a}_{2}}$.

\smallskip

{\bf Case 1.} The ideal $\mathfrak{a}_{j}$ is non-principal. For otherwise, $\mathfrak{a}_{j} = \left\langle\displaystyle\frac{\alpha_{j} + \beta_{j}\sqrt{-D}}{2}\right\rangle$ for some integers $\alpha_{j}$ and $\beta_{j}$ of same parity. Since $\mathfrak{a}_{j}^{p} = \left\langle\displaystyle\frac{X_{j} + \sqrt{-D}}{2}\right\rangle$, it follows that $\beta_{j} \neq 0$. Therefore, $$Y_{j}^{p} = N(\mathfrak{a}_{j}^{p}) = \left(N\left(\displaystyle\frac{\alpha_{j} + \beta_{j}\sqrt{-D}}{2}\right)\right)^{p} = \left(\displaystyle\frac{\alpha_{j}^{2} + \beta_{j}^{2}D}{4}\right)^{p} \geq \left(\displaystyle\frac{1 + D}{4}\right)^{p} > (ab^{\frac{p - 1}{2}}q^{\frac{p + 1}{2}})^{p} > Y_{j}^{p},$$ which is a contradiction. Therefore, $\mathfrak{a}_{j}$ is non-principal for $j = 1,2$.

\smallskip

{\bf Case 2.} For $\mathfrak{b} = \mathfrak{a}_{2}^{t}$ or $\overline{\mathfrak{a}_{2}}^{t}$ and $t$ even, the ideal $\mathfrak{a}_{1}\mathfrak{b}$ is non-principal. Otherwise, $\mathfrak{a}_{1}\mathfrak{b}^{t} = \left\langle\displaystyle\frac{\alpha + \beta\sqrt{-D}}{2} \right\rangle$ for some integers $\alpha$ and $\beta$ of same parity. Then
\begin{eqnarray*}
\mathfrak{a}_{1}^{p}\mathfrak{b}^{pt} &=& \left\langle\displaystyle\frac{X_{1} + \sqrt{-D}}{2}\right\rangle\left\langle\displaystyle\frac{X_{2} \pm \sqrt{-D}}{2}\right\rangle^{t}\\ &=& \left\langle\displaystyle\frac{X_{1} + \sqrt{-D}}{2}\right\rangle \left\langle\displaystyle \frac{\left(\displaystyle\sum_{i = 0}^{t}{t \choose i}X_{2}^{i}(\pm\sqrt{-D})^{t - i}\right)^{t}}{2^{t}}\right\rangle
\end{eqnarray*}
The coefficient of $\sqrt{-D}$ is $$A := X_{1}\left(\displaystyle\sum_{i ~{\rm{odd}}}{t \choose i}X_{2}^{t - i}(\pm\sqrt{-D})^{i - 1}\right) + \displaystyle\sum_{i ~{\rm{even}}}{t \choose i}X_{2}^{t - i}(\pm\sqrt{-D})^{i}.$$
Since $t$ is even, we write $t = 2t_{1}$ for an integer $t_{1}$. Then $A = 0$ implies that $X_{2}$ divides $D^{t_{1}}$, which in turn implies that $\gcd(X_{2},D) > 1$. This, together with the equation $X_{2}^{2} + D = 4Y_{2}^{p}$, implies that $\gcd(X_{2},Y_{2}) > 1$, which is a contradiction to the assumption that $D$ is square-free. Consequently, we get that $A \neq 0$ and therefore, $\beta \neq 0$. Hence $$Y_{1}^{p}Y_{2}^{pt} = N(\mathfrak{a}_{1}^{p}\mathfrak{b}^{pt}) = \left(\displaystyle\frac{\alpha^{2} + D\beta^{2}}{4}\right)^{p} \geq \left(\displaystyle\frac{1 + D}{4}\right)^{p} > \left(ab^{\frac{p - 1}{2}}q^{\frac{p + 1}{2}}\right)^{p},$$ which is a contradiction.

\smallskip

{\bf Case 3.} For $\mathfrak{b} = \mathfrak{a}_{2}^{t}$ or $\overline{\mathfrak{a}_{2}}^{t}$ and $t$ odd, the ideal $\mathfrak{a}_{1}\mathfrak{b}$ is non-principal. Otherwise, $\mathfrak{a}_{1}\mathfrak{b}^{t} = \left\langle\displaystyle\frac{\alpha + \beta\sqrt{-D}}{2} \right\rangle$ for some integers $\alpha$ and $\beta$ of same parity. Then, similar as before, the coefficient of $\sqrt{-D}$ in $\left\langle\displaystyle\frac{X_{1} + \sqrt{-D}}{2}\right\rangle\left\langle\displaystyle\frac{X_{2} \pm \sqrt{-D}}{2}\right\rangle^{t}$ is $$B := X_{1}\left(\displaystyle\sum_{i ~{\rm{odd}}}{t \choose i}X_{2}^{t - i}(\pm \sqrt{-D})^{i - 1}\right) + \displaystyle\sum_{i ~{\rm{even}}}{t \choose i}X_{2}^{t - i}(\pm\sqrt{-D})^{i}.$$ Since $t$ is odd, we write $t$ as $t = 2t_{1} + 1$. Now, if $B = 0$, then we obtain that $X_{2}$ divides $X_{1}D^{m}$. The assumption that $D$ is square-free yields that $\gcd(X_{2},D) = 1$. Thus $X_{2}$ divides $X_{1}$.

\smallskip

Now, we prove that $X_{2}$ is either $p$ or $-p$. For that, let $\ell$ be a prime divisor of $X_{2}$. Then $\ell$ is odd. Moreover, since $X_{2} \mid X_{1}$, we obtain that $X_{2} \mid (X_{1} \pm X_{2})$. That is, $X_{2} \mid g(a,b)$ and $X_{2} \mid (a - b)q^{p}$. Consequently, $\ell \mid g(a,b)$ and $\ell \mid (a - b)q^{p}$. We note that if $\ell = q$, then $\ell^{2} \mid f_{q}(a,b) = D$, a contradiction to the assumption that $D$ is square-free. Consequently, $\ell \neq q$ and hence $a \equiv b \pmod {\ell}$. Again, $\ell \nmid ab$, because $D$ is square-free. Therefore, $0 \equiv g(a,b) \equiv pa^{p - 1} \pmod {\ell}$ implies that $\ell = \pm p$. Thus $X_{2} = \pm p^{k}$, for some integer $k \geq 1$. Therefore, the congruence $a \equiv b \pmod {X_{2}}$ translates to $a \equiv b \pmod {p^{k}}$. Consequently, we have $0 \equiv g(a,b) \equiv pa^{p - 1} \pmod {p^{k}}$. Since $p \nmid a$, this forces $k = 1$. Now, we prove that the equation $X_{2} = \pm p$ has no solution in integers for $a,b \in [\frac{1}{2p}q^{\frac{p}{p - 2}},(\frac{1}{2p} + \frac{1}{2^{p}p^{p}})q^{\frac{p}{p - 2}}] \cap \mathbb{Z}$.

\smallskip

If possible, suppose that $X_{2} = a^{p - 1} + a^{p - 2}b + \ldots + ab^{p - 2} + b^{p - 1} - (a - b)q^{p} = \pm p$ has an integral solution in the aforementioned range. Then we have
\begin{eqnarray*}
p &=& |(a^{p - 1} + a^{p - 2}b + \ldots + ab^{p - 2} + b^{p - 1}) - (a - b)q^{p}|\\ &\geq & |a^{p - 1} + a^{p - 2}b + \ldots + ab^{p - 2} + b^{p - 1}| - |(a - b)q^{p}|\\ &\geq & \left(p\left(\frac{1}{2p}\right)^{p - 1} - \left(\frac{1}{2p}\right)^{p}\right)q^{\frac{p(p - 1)}{p - 2}} = \frac{2p^{2} - 1}{(2p)^{p}}\cdot q^{\frac{p(p - 1)}{p - 2}}.
\end{eqnarray*}
Since $X$ is a large enough positive real number and $q \asymp X^{\frac{p - 2}{2p(p - 1)}}$, we obtain that $$ p \geq \frac{2p^{2} - 1}{(2p)^{p}}\cdot q^{\frac{p(p - 1)}{p - 2}} > p,$$ which is a contradiction.

\smallskip

Hence both $[\mathfrak{a}_{1}]$ and $[\mathfrak{a}_{2}]$ generate a subgroup of order $p$ inside $Cl_{\mathbb{Q}(\sqrt{-D})}$ and they intersect trivially. Therefore, $Cl_{\mathbb{Q}(\sqrt{-D})}$ contains a subgroup isomorphic to $\mathbb{Z}/p\mathbb{Z} \times \mathbb{Z}/p\mathbb{Z}$. This completes the proof of the proposition.
\end{proof}

\section{Square-free values of $f_{q}$}

In this section, first we prove two lemmas about the square-free values of the polynomial $f_{q}$ of Proposition \ref{propn-1}. we prove that $f_{q}$, considered as polynomial in two variables over $\mathbb{Q}$, is square-free in the polynomial ring $\mathbb{Q}[X,Y]$. We also prove that $f_{q}(a,b)$ is not divisible by the square of a fixed prime number for all $(a,b) \in \mathbb{N} \times \mathbb{N}$. This is necessary when we apply Poonen's result (cf. \cite{poonen}) on the square-free values of multivaribale polynomials in our context. 

\begin{lemma}\label{lem-1}
For a prime number $q \geq 5$, the polynomial $$f_{q}(X,Y) = 2(X^{p} + Y^{p})q^{p} - (X - Y)^{2}q^{2p} - (X^{p - 1} + X^{p - 2}Y + \ldots + Y^{p - 1})^{2} \in \mathbb{Z}[X,Y]$$ is square-free as an element of $\mathbb{Q}[X,Y]$. 
\end{lemma}
\begin{proof}
Assume that $f_{q}(X,Y)$ is not square-free as an element of $\mathbb{Q}[X,Y]$. Then we can write $f_{q}(X,Y) = G(X,Y)^{2}H(X,Y)$ for some polynomials $G$ and $H$. We may also assume that $G$ is irreducible in $\mathbb{Q}[X,Y]$. Differentiating the equation with respect to $X$ and $Y$, we see that $\frac{\partial f_{q}}{\partial X} = G^{2}\frac{\partial H}{\partial X} + 2G(\frac{\partial G}{\partial X})H$ and $\frac{\partial f_{q}}{\partial Y} = G^{2}\frac{\partial H}{\partial Y} + 2G(\frac{\partial G}{\partial Y})H$. That is, $G(X,Y)$ is a common divisor of $f_{q},\frac{\partial f_{q}}{\partial X}$ and $\frac{\partial f_{q}}{\partial Y}$. Therefore, $G(X,Y)$ also divides $X\frac{\partial f_{q}}{\partial X} + Y\frac{\partial f_{q}}{\partial Y} = (p - 2)[(g(X,Y)^{2} - (X - Y)^{2}q^{2p}] + pf_{q}(X,Y)$ and consequently, $G(X,Y)$ divides $(g(X,Y)^{2} - (X - Y)^{2}q^{2p}$. Since $G$ is irreducible, $G(X,Y)$ divides $[g(X,Y) - (X - Y)q^{2p}]$ or $[g(X,Y) + (X - Y)q^{2p}]$. This implies that
\begin{equation}\label{lem1eq1}
G(X,Y)^{2} \mid [g(X,Y)^{2} + (X - Y)^{2}q^{2p} \pm 2(X^{p} - Y^{p})q^{p}].
\end{equation}
Again, 
\begin{equation}\label{lem1eq2}
G(X,Y)^{2} \mid -f_{q}(X,Y) = g(X,Y)^{2} + (X - Y)^{2}q^{2p} - 2(X^{p} + Y^{p})q^{p}.
\end{equation}
From \eqref{lem1eq1} and \eqref{lem1eq2}, it follows that $G(X,Y)^{2} \mid 4X^{p}q^{p}$  or $G(X,Y)^{2} \mid 4Y^{p}q^{p}$. Hence $G$ is a power of $X$ or a power of $Y$. This implies that $f_{q}(X,0) \equiv 0$ or $f_{q}(0,Y) \equiv 0$, which is clearly false. This completes the proof of the lemma. 
\end{proof}

\begin{lemma}\label{lemma-2}
Let $q > p$ be a prime number and let $f_{q}(X,Y)$ be as in Lemma \ref{lem-1}. Then there is no prime number $\ell$ such that $\ell^{2}$ divides $f_{q}(a,b)$ for all $(a,b) \in \mathbb{N} \times \mathbb{N}$.
\end{lemma}

\begin{proof}
If possible, suppose that $\ell$ is a prime number number such that $\ell^{2}$ divides $f_{q}(a,b)$ for all $(a,b) \in \mathbb{N} \times \mathbb{N}$. In particular, $\ell^{2} \mid f_{q}(a,a) = a^{p}(4q^{p} - p^{2}a^{p - 2})$. For $a = 1$, this yields $4q^{p} \equiv p^{2} \pmod {\ell^{2}}$. Since $p$ is odd, this implies that $\ell$ is also odd. Moreover, $f_{q}(1,1) \equiv 0 \pmod {\ell^{2}}$ implies that $\ell \neq p$ and $\ell \neq q$.

\smallskip

Now, let $a$ be a primitive root modulo $\ell^{2}$. Then the congruence $f_{q}(a,a) \equiv 0 \pmod {\ell^{2}}$ boils down to $4q^{p} \equiv p^{2}a^{p - 2} \pmod {\ell^{2}}$. Using the congruence $4q^{p} \equiv p^{2} \pmod {\ell^{2}}$, we obtain $p^{2} \equiv p^{2}a^{p - 2} \pmod {\ell^{2}}$. Since $\ell \neq p$, this gives $a^{p - 2} \equiv 1 \pmod {\ell^{2}}$. Since $a$ is a primitive root modulo $\ell^{2}$, it immediately follows that $\ell(\ell - 1) \mid (p - 2)$, which is impossible because $p - 2$ is odd and $\ell(\ell - 1)$ is even. This completes the proof of the lemma.
\end{proof}

Our objective is to estimate the quantity $\displaystyle\sum_{\substack{D = f_{q}(a,b) \\ D \leq X}}\mu^{2}(D)$. In other words, we wish to count the number of times the bivariate polynomial $f_{q}$ assumes a square-free value in a given range. In view of this, we quote a result of Poonen \cite{poonen} about the density of square-free values of a multivarible polynomial with integer coefficients, under the assumption of the $abc$- conjecture. Before stating the result, let us recall a few notations, as used in \cite{poonen}, as follows.

\smallskip

For positive real numbers $B_{1},\ldots ,B_{n}$, we $${\rm{Box}} := {\rm{Box}}(B_{1},\ldots ,B_{n}) := \{(a_{1},\ldots ,a_{n}) \in \mathbb{Z}^{n} : 0 < a_{i} \leq B_{i} \mbox{ for all } i\}.$$

For a set $\mathscr{S} \subseteq \mathbb{Z}^{n}$, define $$\overline{\mu_{n}}(\mathscr{S}) := \underset{B_{1},\ldots ,B_{n - 1} \to \infty}{\lim \sup} \underset{B_{n} \to \infty}{\lim \sup} \displaystyle\frac{\#(\mathscr{S} \cap {\rm{Box}})}{\#{\rm{Box}}} \mbox{ and } \underline{\mu_{n}}(\mathscr{S}) := \underset{B_{1},\ldots ,B_{n - 1} \to \infty}{\lim \inf} \underset{B_{n} \to \infty}{\lim \inf} \displaystyle\frac{\#(\mathscr{S} \cap {\rm{Box}})}{\#{\rm{Box}}}.$$ If $\overline{\mu_{n}}(\mathscr{S}) = \underline{\mu_{n}}(\mathscr{S})$, then $\mu_{n}(\mathscr{S})$ is defined to be the common value. Again, $$\overline{\mu}_{\rm{weak}}(\mathscr{S}) := \underset{\sigma}{\max}\underset{B_{\sigma(1) \to \infty}}{\lim \sup}\ldots \underset{B_{\sigma(n)}\to \infty}{\lim \sup} \displaystyle\frac{\#(\mathscr{S} \cap {\rm{Box}})}{\#{\rm{Box}}},$$ where $\sigma$ runs through all the permutations of the set $\{1,\ldots ,n\}$. Similarly, we define $$\underline{\mu}_{\rm{weak}}(\mathscr{S}) := \underset{\sigma}{\max}\underset{B_{\sigma(1) \to \infty}}{\lim \inf}\ldots \underset{B_{\sigma(n)}\to \infty}{\lim \inf} \displaystyle\frac{\#(\mathscr{S} \cap {\rm{Box}})}{\#{\rm{Box}}}.$$ If $\overline{\mu}_{\rm{weak}}(\mathscr{S}) = \underline{\mu}_{\rm{weak}}(\mathscr{S})$, then $\mu_{\rm{weak}}(\mathscr{S})$ is defined to be the common value. 

\smallskip

We quickly recall the $abc$-conjecture that states that if $a,b$ and $c$ are pairwise relatively prime positive integers with $a + b = c$, then for any $\varepsilon > 0$, there exists a constant $C(\varepsilon) > 0$ such that $c < C(\varepsilon)\left(\displaystyle\prod_{\substack{\ell \mid abc \\ \ell ~ {\rm{prime}}}} p\right)^{1 + \varepsilon}$.

\begin{proposition} (\cite{poonen}, Theorem 3.2 and Corollary 3.3)\label{poonen-lem}
Let $F(X_{1},\ldots ,X_{n}) \in \mathbb{Z}[X_{1},\ldots ,X_{n}]$ be a polynomial that is square-free as an element of $\mathbb{Q}[X_{1},\ldots ,X_{n}]$. Suppose that $X_{n}$ appears in $F$. Let $$\mathscr{S}_{F} := \{\underline{x} \in \mathbb{Z}^{n} : F(\underline{x}) \mbox{ is square-free }\}.$$ Assume the validity of the $abc$-conjecture. Then $\mu_{n}(\mathscr{S}_{F}) = \displaystyle\prod_{\ell}\left(1 - \frac{c_{\ell}}{\ell^{2n}}\right)$, where for each prime number $\ell$, the quantity $c_{\ell}$ stands for the number of $\underline{\alpha} \in (\mathbb{Z}/\ell^{2}\mathbb{Z})^{n}$ satisfying $F(\underline{\alpha}) = 0$ in $\mathbb{Z}/\ell^{2}\mathbb{Z}$. Moreover, if $X_{n}$ does not appear in $F$, then $\mu_{\rm{weak}}(\mathscr{S}_{F}) = \displaystyle\prod_{\ell}\left(1 - \frac{c_{\ell}}{\ell^{2n}}\right)$.
\end{proposition}

\begin{rmk}\label{rmk-1}
Note that $f_{q}(X,Y)$ satisfies the hypotheses of Proposition \ref{poonen-lem}. Also, Lemma \ref{lemma-2} asserts that the square of no fixed prime number divides $f_{q}(a,b)$ for all $a,b \in \mathbb{N}$. Therefore, each term in the Euler product $\displaystyle\prod_{\ell}\left(1 - \frac{c_{\ell}}{\ell^{2n}}\right)$ is non-zero for $f_{q}$. Moreover, Poonen also proved in \cite{poonen} that the quantity $c_{\ell}$ in Proposition \ref{poonen-lem} satisfies $c_{\ell} = O(\ell^{2n - 2})$, using techniques from algebraic geometry. Therefore, each term in the Euler product $\displaystyle\prod_{\ell}\left(1 - \frac{c_{\ell}}{\ell^{2n}}\right)$ is $1 + O\left(\frac{1}{\ell^{2}}\right)$ and hence the Euler product for $f_{q}$ converges to a non-zero constant.
\end{rmk}

\section{Proof of Theorem \ref{main-th}}

Recall that for a large positive real number $X$, let $q$ is a prime number with $q \asymp X^{\frac{p - 2}{2p(p - 1)}}$. That is, $c_{1}X^{\frac{p - 2}{2p(p - 1)}} < q < c_{2}X^{\frac{p - 2}{2p(p - 1)}}$ for two suitably chosen positive constants $c_{1}$ and $c_{2}$. For any such prime number $q$, we choose integers $a$ and $b$ with $\frac{1}{2p}q^{\frac{p}{p - 2}} < a,b < (\frac{1}{2p} + \frac{1}{2^{p}p^{p}})q^{\frac{p}{p - 2}}$. Then we see that $f_{q}(a,b) > 0$ and $f_{q}(a,b) \asymp q^{\frac{2p(p - 1)}{p - 2}} \asymp X$. For the sake of convenience, let $$\mathcal{S} := \{(a,b) \in \mathbb{Z}^{2} : c_{1}^{\prime}q^{\frac{p}{p - 2}} < a,b < c_{2}^{\prime}q^{\frac{p}{p - 2}}\} ~\mbox{ and }~ \mathcal{S}(D) := \#\{(a,b) \in \mathcal{S} : D = f_{q}(a,b)\}.$$ Also, let $$\mathcal{S}_{1} := \displaystyle\sum_{D}\mu^{2}(D)\mathcal{S}(D) ~\mbox{ and }~ \mathcal{S}_{2} := \displaystyle\sum_{D}\mu(D)^{2}\mathcal{S}(D)^{2}.$$

We note that $\mathcal{S}_{1}$ counts the number of square-free positive integral values of $f_{q}(a,b)$, where $(a,b) \in \mathcal{S}$. By Lemma \ref{lem-1}, we see that $f_{q}(X,Y)$ satisfies the hypotheses of Proposition \ref{poonen-lem}. Therefore, from Proposition \ref{poonen-lem} and Remark \ref{rmk-1}, we conclude that $\mathcal{S}_{1} \gg q^{\frac{p}{p - 2}}\cdot q^{\frac{p}{p - 2}} = q^{\frac{2p}{p - 2}}$.

\smallskip

Now, we proceed to find an upper bound for $\mathcal{S}_{2}$. From the definition of $\mathcal{S}_{2}$, we see that $\mathcal{S}_{2}$ is the number of square-free positive integers $D$ with $D = f_{q}(a,b)$, counted with multiplicity $\mathcal{S}(D)$. Therefore, $\mathcal{S}_{2}$ is bounded above by the number of quadruples $(a_{1},a_{2},b_{1},b_{2})$ such that $f_{q}(a_{1},b_{1}) = f_{q}(a_{2},b_{2})$. After simplifying, this boils down to 
\begin{equation}\label{last-eq}
4q^{p}(a_{1}^{p} - a_{2}^{p}) = (g(a_{1},b_{1}) + (a_{1} - b_{1})q^{p})^{2} - (g(a_{2},b_{2}) + (a_{2} - b_{2})q^{p})^{2}.
\end{equation}

The right-hand side of \eqref{last-eq} can be factorized as $(u - v)(u + v)$, where $u = (g(a_{1},b_{1}) + (a_{1} - b_{1})q^{p})^{2}$ and $v = (g(a_{1},b_{1}) + (a_{1} - b_{1})q^{p})^{2}$. Now, if $a_{1} = a_{2}$, then for a fixed value of $b_{1}$, equation \eqref{last-eq} reduces to a degree $2(p - 1)$ polynomial in the variable $b_{2}$ and thus there are at most $2(p - 1)$ integral solutions for $b_{2}$. Since $a_{1} = a_{2}$ can be chosen in $O(q^{\frac{p}{p - 2}})$ ways and similarly $b_{1}$ can also be chosen in $O(q^{\frac{p}{p - 2}})$ ways, the number of choices for the quadruple $(a_{1},a_{2},b_{1},b_{2})$ in this case is $O(q^{\frac{p}{p - 2}}\cdot q^{\frac{p}{p - 2}}) = O(q^{\frac{2p}{p - 2}})$.

\smallskip

If $a_{1} \neq a_{2}$, we may assume without loss of any generality that $a_{1} > a_{2}$. For a fixed such tuple $(a_{1},a_{2})$, we find by \eqref{last-eq} that $(u - v)(u + v) \neq 0$. Since $4q^{p}(a_{1}^{p} - a_{2}^{p})$ can be factorized into the product of two integers in $\sigma_{0}(4q^{p}(a_{1}^{p} - a_{2}^{p}))$ ways, using the classical result $\sigma_{0}(N) = O(N^{\varepsilon})$ for any $\varepsilon > 0$, we get that $4q^{p}(a_{1}^{p} - a_{2}^{p})$ can be factorized in $O(q^{\varepsilon})$ ways. For each such factorization and fixed values of $a_{1}, a_{2}, u$ and $v$, we have 
\begin{equation}\label{lasteq-1}
a_{1}^{p - 1} + a_{1}^{p - 2}b_{1} + \ldots + b_{1}^{p - 1} + (a_{1} - b_{1})q^{p} = u 
\end{equation}
and
\begin{equation}\label{lasteq-2}
a_{2}^{p - 1} + a_{2}^{p - 2}b_{2} + \ldots + b_{2}^{p - 1} + (a_{2} - b_{2})q^{p} = v. 
\end{equation}

We see that \eqref{lasteq-1} and \eqref{lasteq-2} are degree $p - 1$ polynomials in the variables $b_{1}$ and $b_{2}$, respectively. Consequently, there are at most $p - 1$ solutions in $b_{1}$ and $b_{2}$ of \eqref{lasteq-1} and \eqref{lasteq-2}, respectively. Thus if $a_{1} \neq a_{2}$, then each $a_{1}$ and $a_{2}$ can be chosen in $O(q^{\frac{p}{p - 2}})$ ways, and corresponding to the tuple $(a_{1},a_{2})$, the integers $u$ and $v$ can be chosen in $O(q^{\varepsilon})$ ways. Therefore, the number of choices for the quadruple $(a_{1},a_{2},b_{1},b_{2})$ is $O(q^{\frac{p}{p - 2}}\cdot q^{\frac{p}{p - 2}}\cdot q^{\varepsilon}) = O(q^{\frac{2p}{p - 2} + \varepsilon})$ ways. Hence $\mathcal{S}_{2} \ll q^{\frac{2p}{p - 2} + \varepsilon}$.

\smallskip

Now, using the Cauchy-Schwarz inequality, we obtain 
\begin{equation}\label{cs}
\left(\displaystyle\sum_{\substack{D \leq X \\ \mathcal{S}(D) > 0}}\mu(D)^{2}\right)\left(\displaystyle\sum_{\substack{D \leq X \\ \mathcal{S}(D) > 0}}\mu(D)^{2}\mathcal{S}(D)^{2}\right) \geq \left(\displaystyle\sum_{\substack{D \leq X \\ \mathcal{S}(D) > 0}}\mu(D)^{2}\mathcal{S}(D)\right)^{2}.
\end{equation}

From inequality \eqref{cs} and the estimates of $\mathcal{S}_{1}$ and $\mathcal{S}_{2}$, it follows that $$\displaystyle\sum_{\substack{D \leq X \\ \mathcal{S}(D) > 0}}\mu(D)^{2} \gg q^{\frac{4p}{p - 2} - \frac{2p}{p - 2} - \varepsilon} \gg q^{\frac{2p}{p - 2} - \varepsilon} \gg X^{\frac{1}{p - 1} - \varepsilon}.$$

This, together with Proposition \ref{propn-1}, completes the proof of Theorem \ref{main-th}. $\hfill\Box$

\smallskip

{\bf Acknowledgements.} It is a pleasure for the first author to thank Indian Institute of Technology, Guwahati for providing excellent facilities to carry out this research. He also thanks the National Board of Higher Mathematics (NBHM) for the Post-Doctoral Fellowship (Order No. 0204/16(12)/2020/R\& D-II/10925). The second author thanks MATRICS, SERB for their research grant MTR/2020/000467.


\begin{thebibliography}{9999}

\bibitem{AC}
N. Ankeny and S. Chowla, {\it On the divisibility of the class numbers of quadratic fields}, {\sf Pacific J. Math.}, {\bf 5} (1955), 321-324.

\bibitem{olivia}
O. Beckwith, {\it Indivisibility of class numbers of imaginary quadratic fields}, {\sf Res. Math. Sci.}, {\bf 4} (2017), Paper No. 20, 11 pp.

\bibitem{bilu-luca}
Y. Bilu and F. Luca, {\it Divisibility of class numbers: enumerative approach}, {\sf J. Reine Angew. Math.}, {\sf 578} (2005), 79-91.

\bibitem{byeon}
D. Byeon, {\it Class numbers of quadratic fields $\mathbb{Q}(\sqrt{D})$ and $\mathbb{Q}(\sqrt{tD})$}, {\sf Proc. Amer. Math. Soc.}, {\bf 132} (2004), 3137-3140.

\bibitem{byeon1}
D. Byeon, {\it Real quadratic fields with class number divisible by $5$ or $7$}, {\sf Manu. Math.}, {\bf 120} (2006), 211215.

\bibitem{byeon-rama}
D. Byeon, {\it Imaginary quadratic fields with non-cyclic ideal class group}, {\sf Ramanujan J.}, {\bf 11} (2006), 159-164.

\bibitem{byeon2}
D. Byeon and E. Koh, {\it Real quadratic fields with class number divisible by $3$}, {\sf Manu. Math.}, {\bf 111} (2003), 261-263.

\bibitem{kalyan}
K. Chakraborty and M. Ram Murty, {\it On the number of real quadratic fields with class number divisible by $3$}, {\sf Proc. Amer. Math. Soc.}, {\bf 131} (2002), 41-44.

\bibitem{self-jrms}
J. Chattopadhyay, {\it A short note on the divisibility of class number of real quadratic fields}, {\sf J. Ramanujan Math. Soc.}, {\bf 34} (2019), 389-392.

\bibitem{self-acta}
J. Chattopadhyay and S. Muthukrishnan, {\it On the simultaneous $3$-divisibility of class numbers of triples of imaginary quadratic fields}, {\sf Acta Arith.}, {\bf 197} (2021), 105-110.

\bibitem{self-rama}
J. Chattopadhyay and A. Saikia, {\it Simultaneous indivisibility of class numbers of pairs of real quadratic fields}, {\sf Ramanujan J.}, (2021) https://doi.org/10.1007/s11139-021-00456-1

\bibitem{craig1}
M. Craig, {\it A type of class group for imaginary quadratic fields}, {\sf Acta Arith.}, {\bf 22} (1973), 449-459.

\bibitem{craig2}
M. Craig, {\it A construction for irregular discriminants}, {\sf Osaka Math. J.}, {\bf 14} (1977), 365-402.

\bibitem{gillibert1}
J. Gillibert and A. Levin, {\it Pulling back torsion line bundles to ideal classes}, {\sf Math. Res.
Lett.}, {\bf 19} (2012), 1171-1184.

\bibitem{gillibert2}
J. Gillibert and A. Levin, {\it Elliptic surfaces over $\mathbb{P}^{1}$ and large class groups of number fields}, {\sf Int. J. Number Theory}, {\bf 15} (2019), 2151-2162.

\bibitem{hb1}
D. R. Heath-Brown, {\it Quadratic class numbers divisible by $3$}, {\sf Funct. Approx. Comment. Math.}, {\bf 37} (2007), 203-211.

\bibitem{hb2}
D. R. Heath-Brown, Corrigendum to \cite{hb1}, {\it Funct. Approx. Comment. Math.}, {\bf 43} (2010), 227.


\bibitem{iizuka}
Y. Iizuka, {\it On the class number divisibility of pairs of imaginary quadratic fields}, {\sf J. Number Theory}, {\bf 184} (2018), 122-127.

\bibitem{iwasawa}
K. Iwasawa, {\it A note on class numbers of algebraic number fields}, {\sf Abh. Math. Sem. Univ. Hamburg}, {\bf 20} (1956), 257-258.

\bibitem{kishi-miyake}
Y. Kishi and K. Miyake, {\it Parametrization of the quadratic fields whose class numbers are divisible by three}, {\sf J. Number Theory}, {\bf 80} (2000), 209-217.

\bibitem{kishi-komatsu}
Y. Kishi and T. Komatsu, {\it Imaginary quadratic fields whose ideal class groups have $3$-rank at least three}, {\sf J. Number Theory}, {\bf 170}, (2017), 46-54.

\bibitem{kohnen-ono}
W. Kohnen and K. Ono, {\it Indivisibility of class numbers of imaginary quadratic fields and orders of Tate-Shafarevich groups of elliptic curves with complex multiplication}, {\sf Invent. Math.}, {\bf 135} (1999), 387-398.

\bibitem{komatsu-acta}
T. Komatsu, {\it An infinite family of pairs of quadratic fields $\mathbb{Q}(\sqrt{D})$ and $\mathbb{Q}(\sqrt{mD})$ whose class numbers are both divisible by $3$}, {\sf Acta Arith.}, {\bf 104} (2002), 129-136.

\bibitem{komatsu-ijnt}
T. Komatsu, {\it An infinite family of pairs of imaginary quadratic fields with ideal classes of a given order}, {\sf Int. J. Number Theory}, {\bf 13} (2017), 253-260. 

\bibitem{kulkarni}
A. Kulkarni, {\it An explicit family of cubic number fields with large $2$-rank of the class group}, {\sf Acta Arith.}, {\bf 182} (2018), 117-132.

\bibitem{levin-arxiv}
K. Kulkarni and A. Levin, {\it Hilbert's irreducibility theorem and ideal class groups of quadratic fields}, (arXiv:2111.15582v1) 

\bibitem{levin}
A. Levin, S. Yan and L. Wiljanen, {\it Quadratic fields with a class group of large $3$-rank}, {\sf Acta Arith.}, {\bf 197} (2021), 275-292.

\bibitem{luca}
F. Luca, {\it A note on the divisibility of class numbers of real quadratic fields}, {\sf C. R. Math. Acad. Sci. Soc. R. Can}, {\bf 25} (2003), 71-75.

\bibitem{luca-pacelli}
F. Luca and A. Pacelli, {\it Class groups of quadratic fields of $3$-rank at least $2$: effective bounds}, {\sf J. Number Theory}, {\bf 128} (2008), 796-804.

\bibitem{murty}
R. Murty, {\it Exponents of class groups of quadratic fields},  {\sf Topics in Number Theory. University Park}, {\bf 467} (1999), 229-239.

\bibitem{nagell}
T. Nagell, {\it Uber die Klassenzahl imaginar quadratischer Zahkorper}, {\sf Abh. Math. Seminar Univ. Hamburg}, {\bf 1} (1922), 140-150.


\bibitem{ono}
K. Ono, {\it Indivisibility of class numbers of real quadratic fields}, {\sf Compositio Math.}, {\bf 119} (1999), 1-11.

\bibitem{poonen}
B. Poonen, {\it Squarefree values of multivariable polynomials}, {\sf Duke Math. J.}, {\bf 118} (2003), 353-373.


\bibitem{scholz}
A. Scholz, {\it Über die Beziehung der Klassenzahlen quadratischer Körper zueinander}, {\sf J. Reine Angew. Math.}, {\bf 166} (1932), 201-203.

\bibitem{sound}
K. Soundararajan, {\it Divisibility of class numbers of imaginary quadratic fields}, {\sf J. London Math. Soc.}, {\bf 61} (2000), 681-690.


\bibitem{berger}
P. Weinberger, {\it Real quadratic fields with class numbers divisible by $n$}, {\sf J. Number Theory}, {\bf 5} (1973), 237-241.

\bibitem{yama}
Y. Yamamoto, {\it On unramified Galois extensions of quadratic number fields}, {\sf Osaka J. Math.}, {\bf 7} (1970), 57-76.

\bibitem{yu1}
G. Yu, {\it A note on the divisibility of class numbers of real quadratic fields}, {\sf J. Number Theory}, {\bf 97} (2002), 35-44.

\bibitem{yu2}
G. Yu, {\it Imaginary quadratic fields with class groups of $3$-rank at least $2$}, {\sf Manuscripta Math.}, {\bf 163} (2020), 569-574.









\end{thebibliography}
\end{document}